\documentclass[leqno,11pt,twoside]{amsart}
\usepackage{geometry}

\usepackage{times}

\usepackage[all]{xy} 
\usepackage{amsmath, amssymb, amsfonts, latexsym, mdwlist, amsthm}
\usepackage{caption, subcaption}

\usepackage{url}

\usepackage{pgf,tikz}

\usepackage[bookmarks, colorlinks, breaklinks, pdftitle={Wall-crossing implies Brill-Noether},
pdfauthor={Arend Bayer}]{hyperref}
\hypersetup{linkcolor=blue,citecolor=blue,filecolor=black,urlcolor=blue}

\usetikzlibrary{arrows,chains,
    decorations.pathreplacing,
shapes,%
}

\usepackage{paralist}
\setdefaultenum{(a)}{(i)}{}{}
\usepackage{enumitem} 

\def\C{\ensuremath{\mathbb{C}}}

\def\P{\ensuremath{\mathbb{P}}}

\def\R{\ensuremath{\mathbb{R}}}
\def\Z{\ensuremath{\mathbb{Z}}}

\def\alg{\mathrm{alg}}

\def\ch{\mathop{\mathrm{ch}}\nolimits}

\def\Coh{\mathop{\mathrm{Coh}}\nolimits}

\def\cok{\mathop{\mathrm{cok}}}

\def\dim{\mathop{\mathrm{dim}}\nolimits}

\def\Ext{\mathop{\mathrm{Ext}}\nolimits}
\def\Gr{\mathop{\mathrm{Gr}}\nolimits}
\def\Hilb{\mathrm{Hilb}}
\def\Hom{\mathop{\mathrm{Hom}}\nolimits}

\def\im{\mathop{\mathrm{im}}\nolimits}

\def\NS{\mathop{\mathrm{NS}}\nolimits}

\def\Pic{\mathop{\mathrm{Pic}}}

\def\rk{\mathop{\mathrm{rk}}}
\def\RlHom{\mathop{\mathbf{R}\mathcal Hom}\nolimits}

\def\Sym{\mathop{\mathrm{Sym}}}

\def\td{\mathop{\mathrm{td}}\nolimits}

\def\into{\ensuremath{\hookrightarrow}}
\def\onto{\ensuremath{\twoheadrightarrow}}

\def\blank{\underline{\hphantom{A}}}

\def\Db{\mathrm{D}^{b}}
\def\st{\mathrm{stable}}
\def\VrdH{V^r_d(\abs{H})}
\def\VrdC{V^r_d(C)}

\def\star{(*)}

\newcommand\TFILTB[3]{%
\xymatrix@=1pc{
{0 = {#1}_0} \ar[rr]&&
{{#1}_1} \ar[rr]\ar[ld] &&
{{#1}_2} \ar[r]\ar[ld] &
{\cdots} \ar[r] & { {#1}_{#3-1}} \ar[rr] &&
{{#1}_{#3} = {#1}} \ar[ld]
\\
& *{{#2}_1} \ar@{.>}[ul] &&
{{#2}_2} \ar@{.>}[ul] & &&&
{{#2}_{{#3}}} \ar@{.>}[ul]
}}

\def\abs#1{\left\lvert#1\right\rvert}

\newcommand\stv[2]{\left\{#1\,\colon\,#2\right\}}

\makeatletter
\newtheorem*{rep@theorem}{\rep@title}
\newcommand{\newreptheorem}[2]{%
\newenvironment{rep#1}[1]{%
 \def\rep@title{#2 \ref{##1}}%
 \begin{rep@theorem}}%
 {\end{rep@theorem}}}
\makeatother

\newtheorem{Thm}{Theorem}[section]
\newreptheorem{Thm}{Theorem}
\newtheorem{Prop}[Thm]{Proposition}

\newtheorem{Lem}[Thm]{Lemma}
\newtheorem{PosLem}[Thm]{Positivity Lemma}
\newtheorem{Cor}[Thm]{Corollary}
\newreptheorem{Cor}{Corollary}

\newreptheorem{Con}{Conjecture}

\newtheorem{heur}[Thm]{Heuristic}

\newtheorem{thm-int}{Theorem}

\theoremstyle{definition}
\newtheorem{Def-s}[Thm]{Definition}
\newtheorem{Def}[Thm]{Definition}
\newtheorem{Rem}[Thm]{Remark}

\def\C{\ensuremath{\mathbb{C}}}

\def\P{\ensuremath{\mathbb{P}}}

\def\R{\ensuremath{\mathbb{R}}}
\def\Z{\ensuremath{\mathbb{Z}}}

\def\cA{\ensuremath{\mathcal A}}

\def\cE{\ensuremath{\mathcal E}}

\def\cO{\ensuremath{\mathcal O}}

\def\cT{\ensuremath{\mathcal T}}

\def\cV{\ensuremath{\mathcal V}}

\def\aa{\ensuremath{\mathbf a}}

\def\ss{\ensuremath{\mathbf s}}

\def\vv{\ensuremath{\mathbf v}}
\def\ww{\ensuremath{\mathbf w}}

\usepackage{todonotes}

\def\Halg{H^*_{\alg}(X, \Z)}
\def\Zab{Z_{\alpha, \beta}}
\def\sab{\sigma_{\alpha, \beta}}
\def\osigma{\overline{\sigma}}

\begin{document}

\title[Wall-crossing implies Brill-Noether]{Wall-crossing implies Brill-Noether \\ Applications of stability conditions on surfaces}

\author{Arend Bayer}
\address{School of Mathematics and Maxwell Institute,
University of Edinburgh,
James Clerk Maxwell Building,
Peter Guthrie Tait Road, Edinburgh, EH9 3FD,
United Kingdom}
\email{arend.bayer@ed.ac.uk}
\urladdr{http://www.maths.ed.ac.uk/~abayer/}


\begin{abstract}
Over the last few years, wall-crossing for Bridgeland stability conditions has led to a large number of
results in algebraic geometry, particular on birational geometry of moduli spaces.

We illustrate some of the methods behind these result by reproving
Lazarsfeld's Brill-Noether theorem for curves on K3 surfaces via wall-crossing. We conclude with a survey of
recent applications of stability conditions on surfaces.

The intended reader is an algebraic geometer with a limited working knowledge of derived categories.
This article is based on the author's talk at the AMS Summer Institute on Algebraic Geometry in Utah, July 2015.
\end{abstract}


\maketitle
\setcounter{tocdepth}{1}
\tableofcontents

\section{Introduction}

Merely following the logic of wall-crossing
naturally leads one to reprove Lazarsfeld's Brill-Noether theorem for curves on K3 surfaces. I hope
that explaining this proof will serve to illustrate the methods underlying many of the recent applications of
wall-crossing for Bridgeland stability conditions on surfaces, in particular
to the birational geometry of moduli spaces of sheaves.


To state our concrete goal, let $(X, H)$ be a smooth polarised K3 surface.
\begin{description}
\item[Assumption (*)] $H^2$ divides $H.D$ for all curve classes $D$ on $X$.
\end{description}

\begin{Thm}[\cite{Lazarsfeld:BN-Petri}] \label{thm:BN}
Let $(X, H)$ be a polarised K3 surface satisfying Assumption (*).
Let $C$ be any smooth curve in the linear system
$\abs{H}$. Then the Brill-Noether variety $W^r_d(C)$ has expected dimension $\rho(r, g, d)$; in
particular, it is empty if and only if $\rho(r, g, d) < 0$.
\end{Thm}
Here $g$ is the genus of $C$, $\rho(r, d, g)$ is the Brill-Noether number,
and $W^r_d(C)$ denotes the variety of globally generated degree $d$ line bundles $L$ on $C$
with at least $r+1$ global sections;
see Section \ref{sect:BNintro} for more details. 
This, of course, is closely related to the space of morphisms $C \to \P^r$, and
thus Lazarsfeld's theorem answers one of the most basic questions about the projective geometry of
$C$. The corresponding statements for arbitrary generic curves was famously proved by degeneration 
in \cite{Griffiths-Harris:Brill-Noether}; Lazarsfeld's proof instead uses vector bundles on the K3
surface.

Our theorem is in fact a bit more precise than the statements proved in
\cite{Lazarsfeld:BN-Petri}: we show that $W^r_d(C)$ has expected dimension for \emph{every} curve
$C \in \abs{H}$, while Lazarsfeld's methods only yield this result for \emph{generic} such
$C$.\footnote{I am grateful to Gavril Farkas for explaining to me that this stronger result can,
for a sufficiently general K3 surface, also be proved by degeneration methods. The idea is similar to
the proof of \cite[Theorem 4.4]{explicit-BNP-general}. We degenerate to an elliptic K3 surface
where $H$ admits an effective decomposition $H = R + gE$ with $R^2 = -2, R.E = 1$ and $E^2 = 0$.
Since $h^0(\cO_X(H)) = g$, 
any curve in $\abs{H}$ is a
union of the rational curve $R$ with $g$ elliptic curves of class $E$. Using semistable reduction,
we see that any curve in $\abs{H}$ on our original K3 surface degenerates to a curve obtained from a
rational curve in $\overline{M}_{0, g}$ by attaching an elliptic tail at each marked point. Any linear series
degenerates to a cuspidal linear series on the corresponding rational curve. By the results of
\cite{Eisenbud-Harris:cuspidal-rational}, the Brill-Noether varieties of cuspidal linear series have
expected dimension for \emph{every} curve in $\overline{M}_{0, g}$.}
We discuss the differences between the methods presented here and those of
\cite{Lazarsfeld:BN-Petri} in more detail at the end of Section \ref{sect:firstwall}.

We will prove Theorem \ref{thm:BN} as a consequence of Theorem \ref{thm:VrdC}, which also allows for
singular curves and pure torsion sheaves; see the discussion at the end of Section
\ref{sect:BNintro}.

\subsection*{Background}
Over the last few years, stability conditions and wall-crossing have produced many results
in birational geometry completely unrelated to derived categories; we conclude this article with a
survey of such results.
While this development may have come as
a surprise to many, myself included, it is, as often, quite a logical 
development in hindsight---as well as perhaps in the foresight of a few, more on that below.

There are many famous conjectures (due to Bondal, Orlov, Kawamata, Katzarkov, Kuznetsov and others)
predicting precise relations
between the derived category of a variety and its birational geometry. But below the surface,
wall-crossing is much closer connected to vector bundle techniques as
introduced in the 1980s, and as used in Lazarsfeld's proof. I hope that this direct comparison
will illuminate the additional insights coming from the derived category, stability conditions and wall-crossing.

\subsection*{Intended reader} 
I assume that the reader is an algebraic geometer with a passing
familiarity of basic facts about the bounded derived category $\Db(X) = \Db(\Coh X)$ of coherent sheaves on smooth
projective varieties $X$; for references, the reader may consult \cite[Chapter
10]{Weibel:homological} or \cite[Chapters 1--2]{Huybrechts:FM}.

\subsection*{Omissions and apologies}
This survey does not say anything on stability conditions on higher-dimensional varieties. It is
also ignorant of applications of wall-crossing and stability conditions to Donaldson-Thomas theory
(see \cite{Yukinobu:DTsurvey}), to
the derived category itself (as e.g.~in \cite{K3Pic1}), and of connections to mirror symmetry (see \cite{Bridgeland:spaces}). 

The survey would also like to apologise for not actually giving a definition of stability
conditions (instead it only describes the construction of some stability conditions on a K3 surface).
We refer the interested reader to the original articles \cite{Bridgeland:Stab, Bridgeland:K3}, or to
\cite{Daniel:intro-stability, stability-tour, Emolo-Benjamin:lecture-notes} for surveys.

\subsection*{Acknowledgements} Such a survey may be the right place to try to appropriately
thank Aaron Bertram, who stubbornly convinced me
and others of the power of wall-crossing for questions in birational geometry, and whose foresight
influenced my approach to the topic to great extent. Of course, I am also very much indebted to Emanuele
Macr{\`{\i}}---this article is directly inspired by our joint work, and 
greatly benefitted from a number of additional conversations with him. I am also grateful for comments
by Izzet Coskun, Gavril Farkas, Soheyla Feyzbakhsh, Davesh Maulik and Kota Yoshioka.

The author was supported by ERC starting grant no.~337039 ``WallXBirGeom''.

\subsection*{Plan of the paper}

Sections \ref{sect:heart}, \ref{sect:geomstability} and \ref{sect:modulispaces} review properties of
stability conditions on K3 surfaces and moduli space of stable objects; the key results are
Proposition \ref{prop:Cohbeta}, Theorem \ref{thm:Bridgelandmain} and Theorem \ref{thm:dimmodspace}.
Section \ref{sect:BNintro} recalls the basics about Brill-Noether for curves in K3 surfaces and the
associated moduli space of torsion sheaves. The proof of Theorem \ref{thm:BN} is contained in
Sections \ref{sect:firstwall} and \ref{sect:conclusion}. Section \ref{sect:biratBN} reinterprets
the results as a statement of the birational geometry of the moduli space of torsion sheaves.
Section \ref{sect:biratgeneral} systematically reviews results on birational geometry of moduli
space obtained via wall-crossing, as well as other applications of stability conditions on surfaces.

\section{The heart of the matter}
\label{sect:heart}

The key derived category technique that we need is the construction of a certain abelian
subcategory $\Coh^\beta X \subset \Db(X)$ of two-term complexes, see Proposition \ref{prop:Cohbeta}.
More technically, we construct a \emph{bounded t-structure} which has $\Coh^\beta X$ as its
\emph{heart}.
In addition to $H$, it depends on a choice of real number $\beta$.

We recall the slope of a coherent sheaf $E$, for later convenience shifted
by $\beta \in \R$:
\[ \mu_\beta(E) := \begin{cases} 
\frac{H.c_1(E)}{H^2 \rk(E)} - \beta & \text{if $\rk(E) > 0$,} \\
+\infty & \text{otherwise.} \end{cases}
\]
The following slight modification (which I learned from Yukinobu Toda) of the definition of slope-stability implicitly accounts
correctly for torsion sheaves:
\begin{Def} We say that $E \in \Coh X$ is slope-(semi)stable if for all subsheaves $A \subset E$,
we have 
$\mu_\beta(A) < (\leq) \ \mu_\beta(E/A)$. 
\end{Def}

Every sheaf $E$ has a (unique and functorial\footnote{Given $\mu
\in \R$, let $i$ be maximal such that $\mu_\beta(E_i/E_{i-1}) > \mu$, and set $E^\mu := E_i$. Then
the assignment $E \mapsto E^\mu$ is functorial. In particular, the sheaves $T(E)$ and
$F(E)$ occurring in
Proposition \ref{prop:torsionpair} depend functorially on $E$.}) Harder-Narasimhan (HN) filtration: a sequence 
$  0 = E_0 \into E_1 \into \dots \into E_m = E $
of coherent sheaves where $E_i/E_{i-1}$ is slope-semistable for $1 \le i \le m$, and with 
\[
\mu_\beta^+(E) := \mu_\beta\left(E_1/E_0\right) > \mu_\beta\left(E_2/E_1\right) > \dots
>\mu_\beta^-(E) :=  \mu_\beta\left(E_m/E_{m-1}\right).
\]
Moreover, if $E, F$ are slope-semistable with $\mu_\beta(E) > \mu_\beta(F)$, then $\Hom(E, F) = 0$.

We use the existence of HN filtrations to break the abelian category of coherent sheaves
into two pieces $T^\beta, F^\beta \subset \Coh X$:
\begin{align*}
T^\beta
& = \stv{\cT}{\mu_\beta^-(E) > 0}
= \stv{\cT}{\text{all HN-factors of $\cT$ satisfy $\mu_\beta(\blank) > 0$} \vphantom{\mu_\beta^-} } \\
& = \stv{\cT}{\text{all quotients $\cT \onto \cE$ satisfy $\mu_\beta(\cE) > 0$}\vphantom{\mu_\beta^-} } \\
& = \left\langle \cT \colon \text{$\cT$ is slope-semistable with $\mu_\beta(\cT) >
0$}\vphantom{\mu_\beta^-}  \right\rangle \\
F^\beta
& = \stv{\cT}{\mu_\beta^+(E) \leq 0}
= \stv{\cT}{\text{all HN-factors of $\cT$ satisfy $\mu_\beta(\blank) \le 0$}\vphantom{\mu_\beta^-} } \\
& = \stv{\cT}{\text{all subobjects $\cA \into \cT$ satisfy $\mu_\beta(\cA) \le 0$}\vphantom{\mu_\beta^-} } \\
& = \left\langle \cT \colon \text{$\cT$ is slope-semistable with $\mu_\beta(\cT) \le 0$} \vphantom{\mu_\beta^-} \right\rangle \\
\end{align*}
Here $\langle \cdot \rangle$ denotes the extension-closure, i.e., the smallest subcategory of $\Coh X$ 
containing the given objects and closed under extensions. The equivalence of the above formulations
follows from the existence of Harder-Narasimhan filtrations, and the $\Hom$-vanishing between
stable objects mentioned above. 

The formal properties of this
pair of subcategories can be summarised as follows:

\begin{Prop}  \label{prop:torsionpair}
 The pair $(T^\beta, F^\beta)$ is a torsion pair, i.e.:
\begin{enumerate}
\item For $T \in T^\beta, F \in F^\beta$, we have $\Hom(T, F) = 0$. \label{enum:Homvanishing}
\item Each $E \in \Coh X$ fits into a (unique and functorial) short exact sequence
\[ 0 \to T(E) \to E \to F(E) \to 0 \] 
with $T(E) \in T^\beta$ and $F(E) \in F^\beta$. \label{enum:ses}
\end{enumerate}
\end{Prop}
\begin{proof}
Given a non-zero element $f \in \Hom(T, F)$, we have a surjection $T \onto \im f$ and therefore
$\mu_\beta(\im f) > 0$; but we also have an injection $\im f \into F(E)$ and therefore
$\mu_\beta(\im f) \le 0$. This contradiction proves (\ref{enum:Homvanishing}).

As for (\ref{enum:ses}), consider the HN filtration of $E$, and let $i$ be maximal such that
$\mu_\beta(E_i/E_{i-1}) > 0$; then $T(E) := E_i$ satisfies the claim.
\end{proof}

For us, the most important result on derived categories is the following Proposition; thereafter, all our
arguments will live completely within the newly constructed abelian category.
\begin{Prop}[\cite{Bridgeland:K3, Happel-al:tilting}] \label{prop:Cohbeta}
The following (equivalent) characterisations define an abelian subcategory of $\Db(X)$:
\begin{align*}
\Coh^\beta X
& =  \left\langle T^\beta, F^\beta[1] \right\rangle  \\
& = \stv{E}{H^{0}(E) \in T^\beta, H^{-1}(E) \in F^\beta, H^i(E) = 0 \ \text{for $i \neq 0, -1$}} \\
& = \stv{E}{E \cong F^{-1} \xrightarrow{d} F^0, \ker d \in F^\beta, \cok d \in T^\beta}
\end{align*}
\end{Prop}
Rather than giving a proof, I will try to convey some intuition for the behaviour of this
abelian category. To begin with, short
exact sequences in $\Coh^\beta X$ are exactly those exact triangles
$A \to E \to B \to A[1]$ in $\Db(X)$ for which all of $A, E, B$ are in $\Coh^\beta X$; then
$A$ is the subobject, and $B$ is the quotient. In particular, every object $E \in \Coh^\beta$ fits
into a short exact sequence 
\begin{equation} \label{eqn:obviousses}
H^{-1}(E)[1] \into E \onto H^0(E).
\end{equation}
The isomorphism class of $E$ is determined by the extension class in $\Ext^1(H^0(E), H^{-1}(E)[1]) = \Ext^2(H^0(E), H^{-1}(E))$.

Every short exact sequence 
in $\Coh^\beta X$ gives a six-term long exact sequence in cohomology (with respect to $\Coh X$)
\begin{equation} \label{eqn:longexact}
 0 \to H^{-1}(A) \to H^{-1}(E) \to H^{-1}(B) \to H^0(A) \to H^0(E) \to H^0(B) \to 0
\end{equation}
with $H^{-1}(\blank) \in F^\beta$ and $H^0(\blank) \in T^\beta$.

The following observation already illustrates how closely the abelian category $\Coh^\beta X$ is
related to classical vector bundle techniques.
\begin{Prop} \label{prop:subobject}
Let $E \in T^\beta$, considered as an object of $\Coh^\beta X$. To give a subobject $A \into E$ of
$E$ (with respect to the abelian category $\Coh^\beta X$) is equivalent of giving a sheaf
$A \in T^\beta$ with a map $f \colon A \to E$ whose kernel (as a map of coherent
sheaves) satisfies $\ker f \in F^\beta$.
\end{Prop}
\begin{proof}
Given a subobject $A \into E$, consider the associated long exact cohomology sequence
\eqref{eqn:longexact}. We immediately see that $H^{-1}(A) = 0$, and therefore $A = H^0(A)$ is a
sheaf. The map $f \colon A \to E$, considered as a map of coherent sheaves, has kernel
$\ker f = H^{-1}(B) \in F^\beta$.

Conversely, assume we are given a map $f \colon A \to E$ as specified. Let $B$ be the cone of $f$,
which is the two-term complex with $B^{-1} = A, B^0 = E$, and the differential given by $f$. Then
there is an exact triangle $A \xrightarrow{f} E \to B$. By assumption, $H^{-1}(B) = \ker f \in
F^\beta$; on the other hand, $H^0(B)$ is a quotient of $A \in T^\beta$, and therefore is also in 
$T^\beta$. This shows that $B \in \Coh^\beta X$; hence $A \to E \to B$ is a short exact sequence
and $f$ is injective as a map in $\Coh^\beta X$.
\end{proof}

We conclude this section with a tangential observation on $\Coh^\beta X$. One of the
features of the derived category is that cohomology classes of coherent sheaves become morphisms:
\[ \gamma \in H^k(X, G) = \Hom(\cO_X, G[k]). \]
However, this feature is only useful with additional structures on $\Db(X)$; the abelian category
$\Coh^\beta X$ can precisely play this role. For example, if $k = 1$,
$\beta < 0$ (hence $\cO_X \in T^\beta \subset \Coh^\beta X$) and $G \in F^\beta$, then $\gamma$
becomes a morphism $\cO_X \to G[1]$ in the abelian category $\Coh^\beta X$. This immediately gives additional
methods: one can consider the image of $\gamma$, or one can
try to deduce its vanishing from stability;  see \cite{AB:Reider} for an example of this type of
argument, in this case reproving Reider's theorem for adjoint bundles on surfaces. If instead $k = 2$, then 
$\gamma$ becomes an \emph{extension} $\Ext^1(\cO_X, G[1])$ between two objects within the
same abelian category, and thus produces a corresponding object in $\Coh^\beta X$.
One can, for example, try to determine stability of this object (or study its HN
filtration when it is unstable); see \cite{BBMT:Fujita} for a conjectural application of this idea towards
Fujita's conjecture for threefolds.

\section{Geometric stability}
\label{sect:geomstability}

The goal of this section is to fully explain the meaning of the following result:

\begin{Thm}[\cite{Bridgeland:K3}] \label{thm:Bridgelandmain}
Let $(X, H)$ be a polarised K3 surface. For each $\alpha, \beta \in \R$ with $\alpha > 0$, consider the pair
$\sigma_{\alpha, \beta} := \left(\Coh^\beta X, Z_{\alpha, \beta}\right)$ with $\Coh^\beta X$ as
constructed in Proposition \ref{prop:Cohbeta}, and with $Z_{\alpha, \beta} \colon K(\Db(X)) \to \C$
defined by
\[
Z_{\alpha, \beta}(E) = \left \langle e^{\sqrt{-1} \alpha H + \beta H}, v(E) \right \rangle.
\]
This pair defines a Bridgeland stability condition on $\Db(X)$ if $Re Z_{\alpha, \beta}(\delta) > 0$
for all roots $\delta \in \Halg, \delta^2 = -2$ with $\rk(\delta) > 0$ and $\mu_{H, \beta}(\delta) =
0$; in particular this holds for $\alpha^2 H^2 \ge 2$. 

Moreover, the family of stability conditions $\sigma_{\alpha, \beta}$ varies continuously as
$\alpha, \beta$ vary in $\R_{>0} \times \R$.
\end{Thm}

We begin by explaining the notation. The Mukai vector
of an object $E \in \Db(X)$ given by
\[
v(E) = (v_0(E), v_1(E), v_2(E)) := \ch(E) \cdot \sqrt{\td_X}
	= \left(\ch_0(E), \ch_1(E), \ch_2(E) + \ch_0(E)\right)
\]
lies in the algebraic cohomology $\Halg$.
The pairing $\langle \blank , \blank \rangle$ is the Mukai pairing
\begin{align*}
\langle v(E), v(F) \rangle & = - \chi(E, F) = \sum_i (-1)^i \dim \Hom(E, F[i]) \\
& = v_1(E) v_1(F) - v_0(E) v_2(F) - v_2(E) v_0(F).
\end{align*}
It equips $\Halg$ with the structure of an even lattice of signature $(2, \rho(X))$, where $\rho(X)$
is the Picard rank of $X$. Roots in this lattice are classes $\delta$ with $\delta^2 = -2$.

Explicitly, the central charge $Z_{\alpha, \beta}$ is given by
\begin{align} \label{eqn:centralchargeexplicit}
Z_{\alpha, \beta}(E) =  &
\sqrt{-1} \alpha H \bigl(v_1(E) - \beta H \rk(E) \bigr) \\
& - v_2(E) + \beta H v_1(E) + \frac{\alpha^2 H^2 - \beta^2 H^2}2 v_0(E)  \nonumber
\end{align}

For a sheaf $E$, we have 
$ \Im Z_{\alpha, \beta}(E) \ge 0$ if and only if $\mu_\beta(E) \ge 0$.
Using the short exact sequence \eqref{eqn:obviousses} and $Z_{\alpha, \beta}(F[1]) = - Z_{\alpha,
\beta}(F)$, one can immediately conclude 
\begin{Lem} \label{lem1}
If $E \in \Coh^\beta X$, then $\Im Z_{\alpha, \beta}(E) \ge 0$.
\end{Lem}
In other words, $\Im Z_{\alpha, \beta}$ behaves like a rank function on the abelian category $\Coh^\beta X$: it is a
non-negative function on its set of objects that is additive on short exact sequences. 
We want to define a notion of slope by using the real part $\Re \Zab$ as a degree:
\begin{equation} \label{eq:defnu}
E \in \Coh^\beta X \mapsto \nu_{\alpha, \beta}(E) = \frac{-\Re Z_{\alpha, \beta}(E)}{\Im \Zab(E)}.
\end{equation}
To make this well-behaved, we need one further observation:
\begin{Lem} \label{lem2}
Assume $\alpha, \beta$ satisfy the assumptions of Theorem \ref{thm:Bridgelandmain}.
If $0 \neq E \in \Coh^\beta X$ satisfies $\Im \Zab(E) = 0$, then $\Re \Zab(E) < 0$.
\end{Lem}
\begin{proof}
The short exact sequence \eqref{eqn:obviousses} shows that
\[ \Im \Zab\left(H^{-1}(E)\right) = 0 = \Im \Zab \left(H^0(E)\right).
\]
It follows that if $H^0(E)$ is non-trivial, then it is a zero-dimensional torsion sheaf, in which
case $\Re \Zab (H^0(E)) = -\ch_2(E) < 0$. If $H^{-1}(E) \neq 0$, then it must be a slope-semistable
sheaf with $\mu_\beta \left(H^{-1}(E)\right) = 0$. It is enough to consider the case that it is
\emph{stable}. Then with $\vv := v\left(H^{-1}(E)\right)$ we have $\vv^2 \ge -2$ by
Hirzebruch-Riemann-Roch and Serre duality. If $\vv^2 = -2$,
the claim follows from our assumptions on $\alpha, \beta$. Otherwise, if $\vv^2 \ge 0$, then
\begin{align*}
\Re \Zab(\vv) = \left\langle \Re e^{\sqrt{-1}\alpha H} , e^{-\beta H} \vv \right\rangle = 
\left\langle \left(1, 0, -\frac{\alpha^2H^2}2\right), \left(r, 0, s\right) \right\rangle
=  -s + \frac{\alpha^2H^2 \cdot r}2
\end{align*}
Since $r > 0$ and $-2rs = \left(e^{-\beta H} \vv\right)^2 = \vv^2 \ge 0$ we have
\[ \Re \Zab\left(H^{-1}(E)\right) = - \Re \Zab(\vv) < 0 \]
proving the claim.
\end{proof}

This finally leads to a notion of stability for objects in $\Db(X)$: we say that $E \in \Db(X)$ is 
$\sigma_{\alpha, \beta}$-semistable if some shift $E[k]$ is contained in the abelian category
$\Coh^\beta X$, and if that object $E[k]$ is slope-semistable with respect to the slope-function
$\nu_{\alpha, \beta}$. We need one more result to conclude that $(\Coh^\beta X, \Zab)$ is a
stability condition (which we state without proof):
\begin{Prop} Any $E \in \Coh^\beta X$ admits a HN filtration: a filtration whose quotients are
$\nu_{\alpha, \beta}$-semistable objects of decreasing slopes.
\end{Prop}

Finally, we need to explain what we mean by a \emph{continuous} family of stability conditions. The
technical underlying notion here is the \emph{support property}; it implies that for small variations 
of the central charge $Z$, the variation of the phases $\phi(\blank) = \frac 1\pi \cot^{-1}
\nu(\blank)$ of all semistable objects can be bounded 
simultaneously. What we need is the following consequence:

\begin{Cor} \label{cor:wallcrossing}
Given a class $\vv \in \Halg$, 
there is a chamber decomposition induced by a locally finite set of walls in $\R \times \R_{>0}$
with the following property: for objects of Mukai vector
$\vv$, being $\sab$-stable (or semistable) is independent on the choice of $(\beta, \alpha)$ in any
given chamber.
\end{Cor}

\begin{figure}[!tbp]
 \begin{subfigure}[b]{0.4\textwidth}

\begin{tikzpicture}[line cap=round,line join=round,>=triangle 45,x=1.0cm,y=1.0cm]
\draw[->,color=black] (-4.3,0) -- (3.3,0);
\foreach \x in {-4,-3,-2,-1,1,2,3}
\draw[shift={(\x,0)},color=black] (0pt,2pt) -- (0pt,-2pt) node[below] {\footnotesize $\x$};
\draw[->,color=black] (0,-0.5) -- (0,3.3);
\foreach \y in {,1,2,3}
\draw[shift={(0,\y)},color=black] (2pt,0pt) -- (-2pt,0pt) node[left] {\footnotesize $\y$};
\draw[color=black] (0pt,-10pt) node[right] {\footnotesize $0$};
\clip(-4.3,-0.5) rectangle (3.3,3.3);
\draw [shift={(-2.5,0)}] plot[domain=0:3.141,variable=\t]({1*1.5*cos(\t r)+0*1.5*sin(\t
r)},{0*1.5*cos(\t r)+1*1.5*sin(\t r)});
\draw [shift={(-2.46,0)}] plot[domain=0:3.141,variable=\t]({1*0.92*cos(\t r)+0*0.92*sin(\t
r)},{0*0.92*cos(\t r)+1*0.92*sin(\t r)});
\draw [shift={(-2.43,0)}] plot[domain=0:3.141,variable=\t]({1*0.43*cos(\t r)+0*0.43*sin(\t
r)},{0*0.43*cos(\t r)+1*0.43*sin(\t r)});
\draw (-0.5,0) -- (-0.5,3.3);
\draw [shift={(1.5,0)}] plot[domain=0:3.141,variable=\t]({1*1.5*cos(\t r)+0*1.5*sin(\t r)},{0*1.5*cos(\t
r)+1*1.5*sin(\t r)});
\draw [shift={(1.46,0)}] plot[domain=0:3.141,variable=\t]({1*0.92*cos(\t r)+0*0.92*sin(\t
r)},{0*0.92*cos(\t r)+1*0.92*sin(\t r)});
\draw [shift={(1.43,0)}] plot[domain=0:3.141,variable=\t]({1*0.43*cos(\t r)+0*0.43*sin(\t
r)},{0*0.43*cos(\t r)+1*0.43*sin(\t r)});
\draw[color=black] (0.25,3.2) node {$\alpha$};
\draw[color=black] (3.2,0.3) node {$\beta$};
\end{tikzpicture}

\end{subfigure}
\hfill
 \begin{subfigure}[b]{0.4\textwidth}

\begin{tikzpicture}[line cap=round,line join=round,>=triangle 45,x=1.0cm,y=1.0cm]
\clip(-1.1,-0.6) rectangle (3.8,5.4);
\draw(1.36,2.64) circle (2.02cm);
\draw [domain=0.58:3.8] plot(\x,{(-2.78--4.72*\x)/0.76});
\draw [domain=0.58:3.8] plot(\x,{(-2.71--4.5*\x)/1.69});
\draw [domain=0.58:3.8] plot(\x,{(-1.82--2.96*\x)/1.72});
\draw [domain=0.58:3.8] plot(\x,{(-1.4--2.18*\x)/2.22});
\draw (0.58,-0.06) -- (0.58,5.4);
\draw [domain=-1.1:0.58] plot(\x,{(-2.39--4.19*\x)/-0.59});
\draw [domain=-1.1:0.58] plot(\x,{(-1.9--3.39*\x)/-1.12});
\fill [color=black] (1.34,4.66) circle (1.5pt);
\draw[color=black] (1.61,4.9) node {$\alpha \gg 0$};
\fill [color=black] (0.58,-0.06) circle (1.5pt);
\draw[color=black] (0.94,-0.26) node {$\vv$};
\end{tikzpicture}

\end{subfigure}
\caption{Walls as semi-circles in the upper half plane, or as lines in the projective plane}
\label{fig:wallspictures}
\end{figure}
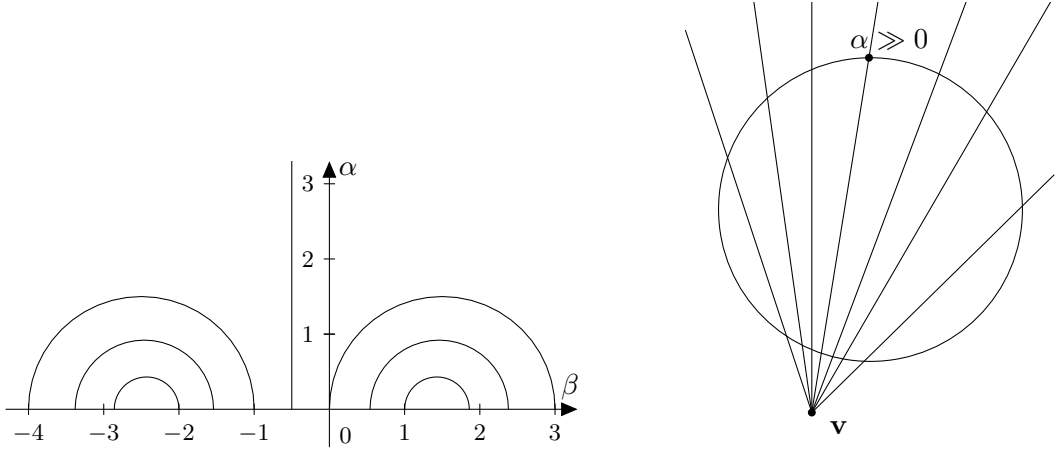

\begin{Rem} \label{rem:wallsaslines}
Locally, such walls are given by the condition that $Z(A)$ and $Z(E)$ are aligned, 
where $E$ is of class $\vv$ and $A \into E$ is a semistable subobject. 
In the $(\beta, \alpha)$-plane, this condition is given by a semi-circle. It is sometimes easier to
visualise the walls if we think of central charges as being characterised by their kernel.
For example, if $\rho(X) = 1$, then the kernel of $\Zab$ is a line inside the negative cone in $\R^3 \cong \Halg \otimes \R$;
our formula for $\Zab$ is the natural identification of the upper half plane with the
projectivization of the negative cone inside $\P^2_\R$. 
The condition that $Z(A)$ and $Z(E)$ are aligned is equivalent to the condition that the kernel is
contained in the rank two sublattice spanned by $v(A)$ and $\vv = v(E)$; thus all walls become lines
in (the image of the negative cone inside) $\P^2_\R$ going through the point corresponding to $\R
\cdot \vv$, see fig.~\ref{fig:wallspictures}. 
\end{Rem}

\section{Moduli spaces of stable objects}
\label{sect:modulispaces}

The most important result on moduli spaces of stable sheaves on K3 surfaces is that they have
expected dimension, and that they are non-empty whenever this dimension is non-negative and the Mukai vector
$\vv$ satisfies some obvious assumptions.
The same holds for moduli spaces of stable objects in the derived category (in
which case we can in fact also drop all assumptions on $\vv$):
\begin{Thm}[Mukai, Yoshioka, Toda, \dots] \label{thm:dimmodspace}
Consider a primitive vector $\vv \in \Halg$, and let $\sigma = \sigma_{\alpha, \beta}$ be a
stability condition that is generic\footnote{This means that $\sigma$ is not on any of the walls for
the wall-and-chamber decomposition described in Corollary \ref{cor:wallcrossing}.}
with respect to $\vv$. Then the coarse moduli space
$M_{\sigma}(\vv)$ of $\sab$-stable objects of Mukai vector $\vv$ exists as a smooth projective
irreducible holomorphic symplectic variety.

It is non-empty iff $\vv^2 \ge -2$, and its dimension is given by $\dim M_{\sigma}(\vv) = \vv^2 + 2$.
\end{Thm}
This is the deepest ingredient in our arguments, and it comes from various sources. The existence
of the moduli space as an algebraic space was proved in \cite{Toda:K3}.
An crucial observation in \cite{Minamide-Yanagida-Yoshioka:wall-crossing} (generalised to our
situation in \cite{BM:projectivity}) uses a Fourier-Mukai transform to reduce all
the statements above to the case of Gieseker-stable sheaves. 
It follows immediately from Serre duality and
Hirzebruch-Riemann-Roch that for $\vv^2 < -2$, the moduli space is empty. If it is non-empty,
Mukai's arguments in \cite{Mukai:Symplectic} shows that the moduli space is smooth and symplectic of
the given dimension. The most difficult statement is the non-emptiness for $\vv^2 \ge -2$; 
its proof uses deformation to elliptic K3 surfaces followed by Fourier-Mukai
transforms to reduce to the case of Hilbert schemes, see \cite[Theorem 8.1]{Yoshioka:Abelian} as well as \cite[Section
2.4]{KLS:SingSymplecticModuliSpaces}.

Assume for simplicity that $M_\sigma(\vv)$ is a fine moduli space, i.e. that it admits a universal
family $\cE \in \Db(M_\sigma(\vv) \times X)$; let $p, q$ denote the projections from the product to 
$M_\sigma(\vv)$ and $X$, respectively. Let $F \in \Db(X)$ be an object with $(\vv, v(F)) = 0$.
Then the determinant line bundle construction $\det$ of \cite{Knudsen-Mumford}
produces a line bundle on $M_\sigma(\vv)$ via
\[
\det \big( p_* \RlHom(\cE, p^* F) \big).
\]
\begin{Thm}[{\cite[Sections 7 and 8]{Yoshioka:Abelian}}] \label{thm:Mukaiisom}
Assume that $\vv$ is primitive with $\vv^2 > 0$, and that $\sigma$ is generic with respect to $\vv$. Then
the determinant line bundle construction induces an isomorphism
\begin{equation} \label{eq:Mukaiisom}
\theta_v \colon \vv^\perp \to \NS M_\sigma(\vv)
\end{equation}
where $\vv^\perp$ denotes the orthogonal complement of $\vv$ inside the algebraic cohomology $\Halg$.
\end{Thm}
We will call $\theta_v$ the \emph{Mukai isomorphism}. 

\begin{Rem} \label{rmk:thetapairing}
In addition, $\theta_v$ identifies the restriction of the Mukai pairing in $\Halg$ with the
\emph{Beauville-Bogomolov-pairing} on the N\'eron-Severi group of the moduli space; however, we will not need that fact
for the proof of Theorem \ref{thm:BN}, only in the concluding sections \ref{sect:biratBN} and
\ref{sect:biratgeneral} on birational geometry of moduli spaces.
\end{Rem}

Consider equations \eqref{eqn:centralchargeexplicit} and \eqref{eq:defnu} for $\alpha \gg 0$; then the slope
$\nu_{\alpha, \beta}(E)$ is approximately given by $- \frac{\alpha} {\mu_{\beta(E)}}$. This observation,
combined with Proposition \ref{prop:subobject} (as well as the bound on Mukai vectors of stable
objects in Theorem \ref{thm:dimmodspace}) leads to the following result:

\begin{Thm} \label{thm:largevolume}
Let $\vv = (v_0, v_1, v_2)$ be a primitive class in $\Halg$ having either positive rank $v_0 > 0$, 
or satisfying $v_0 = 0$ with $v_1$ being effective.
Then there exists $\alpha_0$ such that for all $\alpha \ge \alpha_0$ and all
$\beta < \frac{H.v_1}{H^2 v_0}$ (or $\beta$ arbitrary in case $v_0 = 0$),
the moduli space $M_{\sab}(\vv)$ is equal to the moduli space $M_H(\vv)$ of
$H$-Gieseker-stable sheaves of class $\vv$. More precisely, an object $E \in \Db(X)$ with $v(E) =
\vv$ is 
$\sab$-stable if and only if it is the shift of a Gieseker-stable sheaf.
\end{Thm}

\section{Brill-Noether and the moduli space of torsion sheaves} \label{sect:BNintro}

From now on, let $(X, H)$ be a polarised K3 surfaces satisfying Assumption \star,
and let $d \in \Z$ be a degree.
The natural moduli space related to Brill-Noether is 
$M_H(\vv)$ for $\vv = (0, H, d+1-g)$: it parameterises purely one-dimensional sheaves $F$ of Euler
characteristic $d+1 - g$ whose support $\abs{F}$ is a curve in $\abs{H}$. By \cite{Beauville:ACIS},
the map
\begin{equation} \label{eq:Beauville}
\pi \colon M_H(\vv) \to \abs{H} \cong \P^{g}, \quad F \mapsto \abs{F}
\end{equation}
is a Lagrangian fibration, called the \emph{Beauville integrable system}.
The fibre over a smooth curve $C \subset \abs{H}$ is the Picard variety
$\Pic^d(C)$, and the restriction of the symplectic form to any fibre vanishes. 

We will make all our definitions in the context of $M_H(\vv)$. In particular, let
$T_d(C) = \pi^{-1}(C)$ be the moduli space of pure torsion sheaves supported on $C$ and with Euler
characteristic $d + 1 - g$.
\begin{Def} 
We define the following constructible subsets of $T_d(C)$.
\begin{itemize}
\item $W^r_d(C)$ is the set of globally generated sheaves with at least $r+1$ global sections;
\item $\overline{W}^r_d(C)$ is as above, but without the assumption of being globally generated;
\item $V^r_d(C)$ is the set of sheaves with exactly $r+1$ sections. 
\end{itemize}
In addition, let $\VrdH := \bigcup_{C \in \abs{H}} \VrdC$.
\end{Def}

The expected dimension for each of them is given by the Brill-Noether number
\[
\rho(r, d, g) = g - (r+1)(g-d+r).
\]

Our wall-crossing methods most naturally deal with $\VrdH$ and $V^r_d(C)$; we will prove:
\begin{Thm}\label{thm:VrdC}
Assume $(X, H)$ satisfies Assumption \star, and that  $C \in \abs{H}$ is an arbitrary curve
(possibly singular). If $r, d$ satisfy $0 < d \le g-1$ and
$r \ge 0$, then $\VrdC$ is non-empty if and only if $\rho(r, d, g) \ge 0$, in which case $\dim V^r_d(C) = \rho(r, d, g)$.
\end{Thm}

We will briefly explain how Theorem \ref{thm:VrdC} implies \ref{thm:BN}. 
Since $\rho(r, d, g)$ is a strictly decreasing function of $r$ in our range $d \le
g-1$, and since
\[ \overline{W}^r_d(C) = V^r_d(C) \setminus \bigcup_{r' > r} V^{r'}_d(C),
\]
we conclude $\dim \overline{W}^r_d(C) = \rho(r, d, g)$ for all $d \le g-1$. 
Similarly,
\[ \overline{W}^r_d(C) = W^r_d(C) \cup \bigcup_{d' < d} B_{d'} \]
where $B_{d'}$ parametrises the sheaves whose global sections generate a subsheaf of Euler
characteristic $d' + 1 - g$. Since $C$ is assumed to smooth, we have $B_{d'} \subset
\overline{W}^r_{d'}(C) \times \Sym^{d-d'}(C)$, and thus
\[ \dim B_{d'} \le \rho(r, d', g) + d-d' < \rho(r, d, g), \]
where the last inequality used the assumption $r  \ge 1$ of Theorem \ref{thm:BN}. This 
proves $\dim W^r_d(C) = \rho(r, d, g)$ as claimed. Finally, the case $d > g-1$ follows 
via Serre duality on $C$.

\section{Hitting the wall}
\label{sect:firstwall}

We now consider wall-crossing for the moduli space $M_{\sab}(\vv)$, with $\vv = (0, H, d+1-g)$ as above.
By Theorem \ref{thm:largevolume}, we have
$M_{\sab}(\vv) = M_H(\vv)$ for $\alpha \gg 0$, and we want to find the wall bounding this
Gieseker-chamber.

Consider $\beta = 0$. In this case $\Im Z_{\alpha, 0}(\cO_X) = 0$; by Theorem
\ref{thm:Bridgelandmain}, this means we have stability conditions for
\[ \alpha > \alpha_0 := \sqrt{\frac{2}{H^2}}. \]
For these stability conditions, note that $\cO_X[1]$ is an object of $\Coh^0 X$ with
$\Im Z_{\alpha, 0}(\cO_X[1]) = 0$, i.e. of slope $+\infty$; therefore it is automatically semistable. 
Proposition \ref{prop:subobject} in fact shows that it has no subobjects in $\Coh^0 X$,
and so $\cO_X[1]$ is \emph{stable} for $\beta = 0$. 
(This also shows that the bound of Theorem \ref{thm:Bridgelandmain} is sharp: we have
$Z_{\alpha, 0}(\cO_X[1]) \to 0$ as $\alpha \to \alpha_0$, and the central charge of semistable objects
can never become zero.)


\begin{Lem} \label{lem:nowallshere}
For $\alpha > \alpha_0$ and $\beta = 0$, we have an isomorphism
$M_{\sigma_{\alpha, 0}}(\vv) = M_H(\vv)$ identifying the stable objects with stable sheaves.
\end{Lem}
In other words, there is no wall intersecting the line segment
$\beta = 0, \alpha \in \left(\frac{2}{H^2}, +\infty\right)$.
\begin{proof}
This is a direct consequence of Assumption \star: the objects in $M_H(\vv)$ have ``rank one'' 
in $\Coh^\beta X$, and thus can never be destabilised.

To elaborate, consider equation \eqref{eqn:centralchargeexplicit}. 
We have $\Im Z_{\alpha, \beta= 0}(E) = \alpha H.c_1(E) \in \Z_{\ge 0} \alpha H^2$ for all $E \in
\Coh^0 X$.
Any $L \in M_H(\vv)$ has $\Im Z_{\alpha, \beta= 0}(L) = \alpha H^2$. If $L$ were semistable, each
of its Jordan-H\"older factors $A_i$ would have to have $\Im Z_{\alpha, \beta = 0}(A_i) > 0$
(otherwise it could not have the same slope as $L$), and thus $\Im Z_{\alpha, \beta = 0}(A_i) \ge
\alpha H^2$. This is a contradiction. Combined with Corollary \ref{cor:wallcrossing}, this means
they remain stable along the entire path.
\end{proof}

The key observation linking Brill-Noether to wall-crossing is the following Lemma; the case
$d = g-1$ is one of the first wall-crossings studied in the
literature, see \cite{Aaron-Daniele}.
\begin{Lem} \label{lem:firstwall} 
There is a wall bounding the Gieseker-chamber where $\Zab(\cO_X)$ aligns with $\Zab(\vv)$. 
The sheaves $L \in M_{\sab}(\vv)$ getting destabilised are exactly those with $h^0(L) > 0$, and the destabilising
short exact sequences are given by
\begin{equation} \label{eq:firstwall}
\cO_X^{\oplus h^0(L)} \into L \onto W
\end{equation}
for some object $W$ that remains stable at the wall.
\end{Lem}


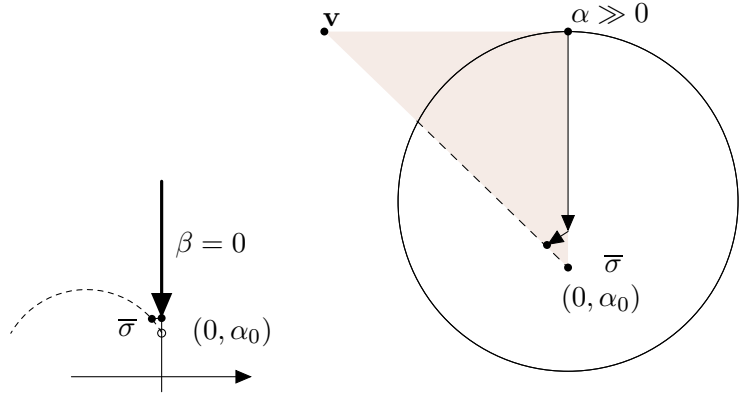
\begin{figure}[!tbp]
 \begin{subfigure}[b]{0.25\textwidth}


\begin{tikzpicture}[line cap=round,line join=round,>=triangle 45,x=1.0cm,y=1.0cm]
\draw[->,color=black] (-1.2,0) -- (1.2,0);
\draw[color=black] (0,-0.2) -- (0,2.6);
\draw [shift={(-1,0)},dash pattern=on 2pt off 2pt]
plot[domain=0.53:2.62,variable=\t]({1*1.16*cos(\t r)+0*1.16*sin(\t r)},{0*1.16*cos(\t
r)+1*1.16*sin(\t r)});
\draw [->,line width=1.2pt] (0,2.6) -- (0,0.78);
\draw (0,0.78) -- (-0.13,0.77);
\draw [color=black] (0,0.58) circle (1.5pt);
\draw[color=black] (0.94,0.6) node {$(0, \alpha_0)$};
\fill [color=black] (0,0.78) circle (1.5pt);
\draw[color=black] (0.65555,1.75) node {$\beta = 0$};
\fill [color=black] (-0.13,0.77) circle (1.5pt);
\draw[color=black] (-0.46,0.66) node {$\overline{\sigma}$};
\end{tikzpicture}


\end{subfigure}
 \begin{subfigure}[b]{0.4\textwidth}

\definecolor{zzttqq}{rgb}{0.6,0.2,0}
\begin{tikzpicture}[line cap=round,line join=round,>=triangle 45,x=1.0cm,y=1.0cm]
\clip(-2.4,0) rectangle (3.5,5.5);
\fill[line width=0pt,color=zzttqq,fill=zzttqq,fill opacity=0.1] (-2.1,4.8) -- (1.14,1.66) --
(1.14,4.8) -- cycle;
\draw(1.14,2.54) circle (2.26cm);
\draw(1.14,2.54) circle (2.26cm);
\draw [dash pattern=on 3pt off 3pt] (-0.87,3.61)-- (1.14,1.66);
\draw [->] (1.14,4.8) -- (1.14,2.14);
\draw [->] (1.14,2.14) -- (0.86,1.96);
\fill [color=black] (1.14,4.8) circle (1.5pt);
\draw[color=black] (1.7,5.04) node {$\alpha \gg 0$};
\fill [color=black] (1.14,1.66) circle (1.5pt);
\draw[color=black] (1.58,1.24) node {$(0, \alpha_0)$};
\fill [color=black] (-2.1,4.8) circle (1.5pt);
\draw[color=black] (-2,4.96) node {$\vv$};
\fill [color=black] (0.86,1.96) circle (1.5pt);
\draw[color=black] (1.74,1.72) node {$\overline{\sigma}$};
\end{tikzpicture}

\end{subfigure}
\caption{From the large volume limit to $\osigma$. No walls in the shaded region!}
\label{fig:pathtoosigma}
\end{figure}

\begin{proof}
This is perhaps most easily explained using the visualisation of walls as lines in the projective
plane discussed in Remark \ref{rem:wallsaslines}. The locus where the central charges of all objects
in  \eqref{eq:firstwall} are aligned is the line segment between $\vv$ and $v(\cO_X)$; in the
upper half-plane picture, it is the arc of a circle ending at $(0, \alpha_0)$. 
Now consider the path in the upper half plane as in fig.~\ref{fig:pathtoosigma} that starts at
$\beta = 0, \alpha \gg 0$, goes straight to a point $(0, \alpha_0 + \epsilon)$ just slightly above
$(0, \alpha_0)$, and then turns left until it hits the above semi-circle. The visualisation of walls via
lines shows immediately that if this path would hit any
other wall beforehand, then that wall would also intersect the straight line segment $\beta = 0,
\alpha \in (\alpha_0, +\infty)$ in contradiction to Lemma \ref{lem:nowallshere}.
Also, $\cO_X$ cannot be destabilised along this path: for $(\beta, \alpha)$ near $(0, \alpha_0)$, we
have $\abs{\Zab(\cO_X)} \ll 1$, and it is the only stable object with that property.

Let $\osigma = (\Coh^{\overline{\beta}} X, \overline{Z})$ be the stability condition at the wall. In
the abelian category of $\osigma$-semistable
objects with central charge aligned with $\overline{Z}(\vv)$, the object $\cO_X$ is a simple object;
hence the natural map $\cO_X^{\oplus h^0(L)} \to L$ must necessarily be an injective map, and the
quotient $W$ must be semistable. 

It remains to prove that $W$ is stable. Note the $\Hom(W, \cO_X)  = 0$ as $W$ is a quotient of $L$.
Moreover, $\Hom(\cO_X, W) = 0$ follows by applying $\Hom(\cO_X, \blank)$ to the short exact sequence
defining $W$. Hence stability of $W$ follows from the following Lemma.
\end{proof}

\begin{Lem} \label{lem:Wstable}
Let $\osigma$ be a stability condition on the wall constructed above.
Let $W$ be an object of class $\vv - t v(\cO_X)$ for some $t \in \Z$, and assume that $W$ is $\osigma$-semistable. Then $W$ is stable
if and only if $\Hom(\cO_X, W) = \Hom(W, \cO_X) = 0$.
\end{Lem}

\begin{figure}
\begin{tikzpicture}[
    media/.style={font={\footnotesize\sffamily}},
    interface/.style={
        postaction={draw,decorate,decoration={border,angle=-45,
                    amplitude=0.3cm,segment length=2mm}}},
    ]

    \fill[gray!10,rounded corners] (0, 0) --  (1, -3) -- (5,-3) -- (5,3) -- (1.5, 3);
    \draw[blue,line width=.5pt,interface](1,-3)--(-1,3);
    \draw[blue](1, -2) node[right]{$\overline Z(\aa)$ aligned with $\overline Z(\vv)$};
    \draw[red,line width=.5pt,interface](-1.5,-3)--(1.5,3);
    \draw[red](1.5, 2) node[right]{$(\ss, \aa) \ge 0$};
    \draw[dashed,gray](0,-3)--(0,3);
    \draw[dashed,gray](-1.5,0)--(5,0);
    \draw[<->,line width=1pt] (1,0) node[above]{$\vv$}-|(0,1) node[above]{$\ss$};

    %
\end{tikzpicture}
\caption{Jordan-H\"older factors of $W$}
\label{fig:ranktwosublattice}
\end{figure}
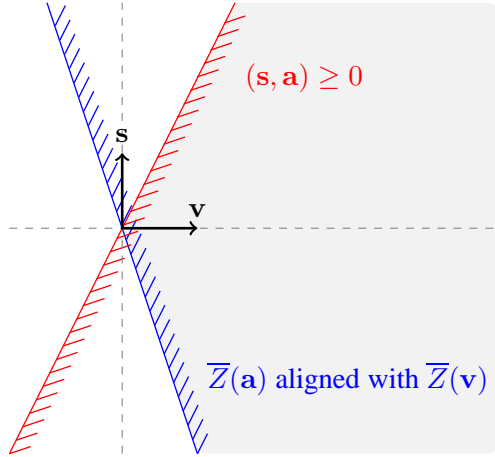

\begin{proof}
Assume first that $\rho(X) = 1$, and consider the Mukai vector $\aa$ of any Jordan-H\"older factor
$A$ of $W$. It must be contained in the rank two sublattice generated by $\vv$ and $\ss:= v(\cO_X)$,
otherwise its central charge would not be aligned with $\overline Z(\vv)$. Further, since $\overline
Z(\aa)$ must be on the same \emph{ray} as $\overline Z(\vv)$, there is a half-plane in this rank two
sublattice containing $\ss, \vv$ and $\aa$, see fig.~\ref{fig:ranktwosublattice}

On the other hand, if $A \neq \cO_X$, then $(\ss, \aa) = - \chi(\cO_X, A) \ge 0$;
since $\ss^2 = -2$ and $(\ss, \vv) > 0$ this cuts out a second half-plane with
configuration as in fig.~\ref{fig:ranktwosublattice}: $\aa$ must lie in the shaded area of the
figure.

It follows that either $\aa = \ss$, or $\aa = a \vv + b \ss$ with $a > 0$. But since
$\ss$ and $\vv$ are a basis for this rank two lattice we must have $a \ge 1$; it follows that
either $a = 1$ or $A = \cO_X$. Hence all but one of the Jordan-H\"older factors of $W$ are
isomorphic to $\cO_X$; so $\cO_X$ is either a subobject or a quotient of $W$, a
contradiction.

When $\rho(X) > 1$, the same arguments apply if we replace all Mukai vectors $\aa = v(A)$ with the
vector $\left(\rk(A), \frac 1{H^2} H.c_1(A), v_2(A)\right) \in \Z^3$; note again that Assumption
\star{} is essential here.
\end{proof}


Let $\ww_r = \vv - (r+1) v(\cO_X) = \left(-(r+1), H, d-g-r)\right)$ be the Mukai vector of $W$.
Note that $\ww_r^2 = 2 \rho(r, d, g) -2$. As in \cite{Lazarsfeld:BN-Petri}, this immediately leads to
the first conclusion:
\begin{Cor}
If $\rho(r, d, g) < 0$, then $V^r_d(C) = \emptyset$ for all $C \in \abs{H}$.
\end{Cor}
\begin{proof}
If $V^r_d(C)$ is non-empty, then by Lemma \ref{lem:firstwall} there exists a $\bar\sigma$-stable object of
class $\ww_r$; by Theorem \ref{thm:dimmodspace} this implies $\ww_r^2 \ge -2$.
\end{proof}
Let us now write $\sigma_+$ for a stability condition on the path of fig.~\ref{fig:pathtoosigma}
just before hitting the wall at $\osigma$.
We will prove a converse to Lemma \ref{lem:firstwall}:

\begin{Lem} \label{lem:plusstable}
Let $W \in M_{\bar\sigma}^{\st}(\ww_r)$ be an object which is $\osigma$-\emph{stable}.
Consider any extension of the form
\[ \cO_X^{r+1} \to E \to W \]
induced by an $(r+1)$-dimensional subspace of $\Ext^1(W, \cO_X)$.
Then $E$ is $\sigma_+$-stable.
\end{Lem}
\begin{proof} Evidently, $E$ is $\overline{\sigma}$-semistable. If $E$ is not $\sigma_+$-stable,
then the destabilising subobject $A \into E$ would necessarily be in the abelian category of
$\overline{\sigma}$-semistable objects of
the same slope as $E$. But since $\cO_X$ and $W$ are simple objects in that category, we can
determine all subobjects of $E$: they are all of the form $\cO_X^d$ for some $d < r+1$. But $\cO_X$
has smaller slope than $E$, a contradiction.
\end{proof}
However, note that by the previous Lemmas, $\sigma_+$ is in the Gieseker-chamber: $M_{\sigma_+}(\vv) = M_H(\vv)$.
Hence such an $E$ is automatically a torsion sheaf in $M_H(\vv)$, with $h^0(E) = r+1$, i.e. $E \in
\VrdH$! To confirm
the existence of such $E$, we need one more result:

\begin{Lem} \label{lem:wrstable}
If $\rho(r, d, g) \ge 0$, then the set of $\bar\sigma$-stable objects in $M_{\sigma_+}(\ww_r)$ is open and non-empty.
\end{Lem}
\begin{proof}
By Theorem \ref{thm:dimmodspace}, the moduli space $M_{\sigma_+}(\ww_r)$ is non-empty of dimension
$\ww_r^2 + 2 = 2\rho \ge 0$; each of its objects are  $\osigma$-semistable of class
$\ww_r$. Consider the Jordan-H\"older factors of such an object $W$. It cannot have $\cO_X$ as a
quotient - otherwise $W$ would not be $\sigma_+$-stable. By Lemma \ref{lem:Wstable}, it is either
stable, or has $\cO_X$ as a subobject.

By induction, it follows just as in Lemma \ref{lem:firstwall} that the Jordan-H\"older filtration of
$W$ is of the form $\cO_X^d \into W \onto W'$, where $W'$ is $\osigma$-stable. 
We want to compute the dimension of the space of such
extension for all such $d$ (if it is non-empty). We write $\ss := v(\cO_X)$ as before, and  set $e
:= (\ww_r, v(\cO_X))$; note $e >0 $. Since $\cO_X$ and $W'$ are $\osigma$-\emph{stable} of the same phase,
we have $\Hom(W', \cO_X) = 0 = \Hom(\cO_X, W')$ and therefore
$\dim \Ext^1(W', \cO_X) = (\ss, v(W'))$. From this, we compute 
the dimension of the space of extensions as
\begin{multline*}
\dim M_{\osigma}^{\mathrm{stable}}(\ww_r - d \ss) + \dim \bigl(\Gr(d, \Ext^1(W', \cO_X)) \bigr) \\
 = \ww_r^2 - 2d e - 2d^2 + 2 + \dim \bigl(\Gr(d, (\ss, \ww_r - d\ss))\bigr) \\
 = \ww_r^2 - 2d e - 2d^2 + 2 + d (e + d)
 = \ww_r^2 + 2 - de - d^2  < \dim M_{\sigma_+}(\ww_r).
\end{multline*}
Therefore, there is an open subset of $M_{\sigma_+}(\ww_r)$ not contained in any of these loci.
\end{proof}

\begin{Cor} \label{cor:dimVrdH}
The set $V^r_d(\abs{H})$ is a Grassmannian-bundle\footnote{When 
$M_{\overline{\sigma}}^\st(\ww_r)$ is a fine moduli space, i.e. it admits a universal family, then
this bundle will be Zariski-locally trivial; in general it will be locally trivial in the \'etale
topology.} 
over $M_{\overline{\sigma}}^{\mathrm{stable}}(\ww_r)$, and its dimension is
\[ \dim V^r_d(\abs{H}) = \rho(r, d, g) + g. \]
\end{Cor}
\begin{proof}
As we already hinted at above, the first statement follows from Lemma \ref{lem:firstwall}, Lemma
\ref{lem:plusstable} (observe that by the long exact cohomology sequence, $E$ as in that Lemma automatically satisfies $h^0(E) = r+1$), and the identification $M_{\sigma_+}(\vv) = M_H(\vv)$. 

By Lemma \ref{lem:wrstable}, this bundle is non-empty. As in the previous Lemma, we can use
stability with respect to $\osigma$ to compute 
\[ \dim \Ext^1(W, \cO_X) = - \chi(W, \cO_X) = (\ww_r, v(\cO_X)) = 2r + 1 + g -d \]
for all $W \in M_{\osigma}^{\mathrm{stable}}(\ww_r)$; the dimension of the
Grassmannian-bundle is therefore
\begin{align*}
\dim V^r_d(\abs{H}) & = \dim M_{\osigma}^{\mathrm{stable}}(\ww_r) + \dim \mathrm{Gr}(r+1, 2r+1+g-d)
\\ & = \ww_r^2+2 + (r+1)(r+g-d) = \rho(r, g,d ) + g.
\end{align*}
\end{proof}


\subsection*{Comparison}
It is quite useful to compare our approach directly to the one taken in \cite{Lazarsfeld:BN-Petri}. 
When the line bundle $L$ is globally generated, then the object $W$ is the shift $M_L[1]$ of the
kernel $M_L$ of the evaluation map $\cO_X^{h^0(L)} \onto L$ (which is surjective as a map of
sheaves, but injective in our abelian category $\Coh^\beta X$), now called Lazarsfeld-Mukai bundle.

Lazarsfeld shows that when all curves in $\abs{H}$ are irreducible, then $M_L$ cannot have
non-trivial endomorphisms. (Otherwise, there would exist an endomorphism $\phi$ of $M_L^\vee$ that
drops rank at some point, and thus everywhere; then he shows that $c_1(\im \phi) + c_1(\cok \phi)$
would be an effective decomposition of $c_1(M_L^\vee) = H$.) This already implies
$v(M_L)^2 \ge -2$, i.e. $\rho \ge 0$ whenever $\VrdH$ is non-empty. 
Now let $P^r_d$ be the $\mathrm{GL}(r+1)$-bundle over $\VrdH$ corresponding to a choice of basis of
global sections.  Lazarsfeld shows directly (with arguments similar to Mukai's arguments behind Theorem
\ref{thm:dimmodspace}) that $P^r_d$ is smooth of expected dimension, and thus the same
hold for $\VrdH$. Combined with the non-emptiness of $\VrdC$ proven in
\cite{Kleiman-Laksov:BN-existence} for all curves, this implies that $\VrdC$ must have expected dimension for generic $C$. 

One new ingredient in our approach, coming directly from stability conditions is that, even without
Assumption \star, it is completely automatic that the object $W$ is $\osigma$-semistable, see the 
proof of Lemma \ref{lem:firstwall}.
The other difference is that wall-crossing gives a global description of $\VrdH$:
the approach in \cite{Lazarsfeld:BN-Petri} is entirely based on an infinitesimal analysis,
in particular, it is agnostic about which vector bundles can appear as $M_L$---in contrast 
to Lemma \ref{lem:plusstable}.

In the next section, we will see that this allows us to both prove the non-emptiness of $\VrdC$ (the
result of \cite{Kleiman-Laksov:BN-existence}) and to determine its dimension (our strengthening of
the result of \cite{Lazarsfeld:BN-Petri}) at the same time, for all $C$: indeed, our Lemma
\ref{lem:nocurves}, and consequently the proof of Theorem \ref{thm:VrdC} at the end of the following
section depend on having a global picture available.

\section{Conclusion}
\label{sect:conclusion}

Corollary \ref{cor:dimVrdH} is a family version of the Brill-Noether theorem in the form of Theorem
\ref{thm:VrdC}. To make conclusions about each individual curve, we will use additional input
from the restriction of the Beauville integrable system \eqref{eq:Beauville}. It gives a map
$\overline{\pi} \colon \VrdH \to \P^g$; and it remains
to prove that all its fibres have the same dimension $\rho(r, d, g) = \dim(\VrdH) - g$.
We will prove this using
fairly standard arguments for maps between holomorphic symplectic varieties, as well as one more
categorical ingredient.

Consider the following diagram of maps:

\begin{equation} \label{maindiagram}
\xymatrix{
\VrdH \ar@{^{(}->}[r] \ar[d]^\phi \ar[rd]^{\overline{\pi}} & M_H(\vv) \ar[d]^{\pi} \\
M_{\bar \sigma}^\st(\ww_r) & \P^g
}
\end{equation}

\begin{Lem} \label{lem:nocurves}
There is no compact curve $D \subset \VrdH$ that is contracted by both $\overline{\pi}$ and by $\phi$.
\end{Lem}
\begin{proof}[Proof (sketch).]
Since the Grassmannian has Picard rank one, all curves contracted by $\phi$ are proportional (in the
group of curves in $M_H(\vv)$ modulo numerical equivalence) to the line in the Grassmannian given as
one of the fibres. If $\overline{\pi}$ were to contract any such curve, it would contract all of
them, and so $\overline{\pi}$ would factor via $\phi$.


There are various ways to see that this is not possible. For example, using the description of the
N\'eron-Severi group of $M_H(\vv)$ in Theorem \ref{thm:Mukaiisom} one can compute both the class
of $L:= \pi^*(\cO_{\P^g}(1)) \in \NS(M_H(\vv))$, and the class of the line $l$ in one of the fibres
of $\phi$; then one sees easily that $L.l \neq 0$. 
Alternatively, the moduli space
$M_{\bar \sigma}^\st(\ww_r)$ contains objects of the form $\cV[1]$ where $\cV$ is a vector bundle whose dual $\cV^\vee$ is
globally generated; that means that varying the extension subspace in 
$\Ext^1(\cV[1], \cO_X) = \Hom(\cO_X, \cV^\vee)$ will result in varying the support of the line
bundle in $\VrdH$.
\end{proof}

Let $\omega$ denote the symplectic form on $M_H(\vv)$.
Recall that $\pi$ is a Lagrangian fibration; in particular, the restriction of $\omega$ to any fibre
$\VrdC$ of $\overline{\pi}$ vanishes. On the other hand:

\begin{Lem} \label{lem:pullback}
The restriction $\omega|_{\VrdH}$ is the pull-back $\phi^* \overline{\omega}$ of the
symplectic form on $M_{\bar \sigma}^\st(\ww_r)$.
\end{Lem}
\begin{proof}
The proof is very similar to arguments in \cite{Mukai:Symplectic}. 

Consider $L \in \VrdH$, and write
$\cO_X^{r+1} \xrightarrow{\alpha} L \xrightarrow{\beta} W$ for the
associated short exact sequence \eqref{eq:firstwall}.  Recall that the tangent space of $M_H(\vv)$ at $L$ is 
$\Hom(L, L[1])$. The subspace tangent to $\VrdH$ are all $f \in \Hom(L, L[1])$ that provide no
obstructions to lifting global sections to the associated extensions. This means
$f \circ \alpha = 0$, or, equivalently, $f = g \circ \beta$ for some $g \in \Hom(W, L[1])$. 
Let $f_W = \beta[1] \circ g \in \Hom(W, W[1])$ denote the associated deformation class of $W$.

A choice of symplectic form on $X$ makes the Serre duality 
$\Hom(A, B) \times \Hom(B, A[2]) \to \C$ on $\Db(X)$ canonical and bi-functorial in
both arguments.
That choice determines the symplectic form on $M_H(\vv)$ using the Serre duality pairing
\[ \Hom(L, L[1]) \times \Hom(L[1], L[2]) \to \C \]
via $\omega(f, f') = \langle f, f'[1] \rangle$; analogously for
$M_{\overline{\sigma}}^{\mathrm{stable}}(\ww_r)$.
Now assume we are given $f, g, f_W$ as above, and $f', g', f'_W$ analogously, see diagram
\eqref{bigdiagram} below for illustration.
We can compute
\begin{align*}
\omega(f, f') & = \langle f, f'[1] \rangle
= \langle g \circ \beta, g'[1] \circ \beta[1] \rangle  
= \langle \beta[1] \circ g \circ \beta, g'[1] \rangle  
= \langle \beta[1] \circ g,  \beta[2] \circ g'[1] \rangle   \\
& = \langle f_W, f'_W \rangle 
= \omega(f_W, f'_W),
\end{align*}
which is precisely the claim.
\begin{equation} \label{bigdiagram}
\xymatrix{
\cO_X^{r+1} \ar[r]^{\alpha} \ar[d]^0 & L \ar[r]^{\beta} \ar[d]^f & W \ar[dl]_g \ar[d]^{f_W} \\
\cO_X^{r+1}[1] \ar[r]^{\alpha[1]} \ar[d]^0 & L[1] \ar[r]^{\beta[1]} \ar[d]^{f'} & W[1]\ar[dl]_{g'}
\ar[d]^{f'_W}  \\
\cO_X^{r+1}[2] \ar[r]^{\alpha[2]} & L[2] \ar[r]^{\beta[2]} & W[2] \\
}
\end{equation}
\end{proof}

\begin{proof}[Proof of Theorem \ref{thm:VrdC}]
By Lemma \ref{lem:nocurves}, we have
\[ \dim \VrdC = \dim \overline{\pi}^{-1}(C) = \dim \phi\left(\overline{\pi}^{-1}(C)\right).
\]
On the other hand, we know that $\VrdC \subset M_H(\vv)$ is isotropic, as it is a subset of the Lagrangian
subvariety $\pi^{-1}(C)$; by Lemma \ref{lem:pullback},
the same holds true for $\phi\left(\VrdC\right) \subset M_{\bar \sigma}^\st(\ww_r)$.

Therefore, 
\begin{equation} \label{eq:upperbound}
 \dim \VrdC = \dim \phi\left(\VrdC\right)  \le \frac 12 \dim M_{\bar \sigma}(\ww_r) = \rho(r, d, g).
\end{equation}
Equality, including the non-emptiness of $\VrdC$, follows from classical results
\cite{Kleiman-Laksov:BN-existence}, but it can also be deduced in our
context. Recall that
\[ \overline{W}^r_d(\abs{H}) = \bigcup_{r' \ge r} V^{r'}_d(\abs{H}) \]
is a closed subvariety of $M_H(\vv)$, and thus projective. Combining Corollary \ref{cor:dimVrdH} with
\eqref{eq:upperbound} (for all $r' \ge r$) we see that in the map $\overline{W}^r_d(\abs{H}) \to \abs{H} \cong \P^g$, all
fibres have at most expected dimension
$\rho(r, g, d) = \dim\overline{W}^r_d(\abs{H}) - \dim \P^g$. Therefore, all fibres have exactly expected
dimension. Again applying the inequality \eqref{eq:upperbound}, this time for all $r' > r$, it follows that
we must have equality.
\end{proof}

\section{Geometry of the Brill-Noether locus and birational geometry of the moduli space}
\label{sect:biratBN}

Our proof of Theorem \ref{thm:BN} in fact provides a geometric description of the Brill-Noether locus
\[ 
\mathrm{BN}_d(\abs{H}) = \stv{L \in M_H(\vv)}{h^0(L) > 0}.
\]
We have shown that in the natural stratification 
\begin{equation} \label{eq:BNstratas}
\mathrm{BN}_d(\abs{H}) = \bigcup_{r \ge 0} \stv{L \in M_H(\vv)}{h^0(L) = r+1} = \bigcup_r \VrdH,
\end{equation}
each stratum is a Grassmannian-bundle of $(r+1)$-dimensional subspaces in a $2r + 1 + g - d$-dimensional
space over a holomorphic symplectic variety of dimension $2 \rho(r, g, d)$. This recovers a result
by Markman \cite{Markman:BNduality} and Yoshioka (\cite[Lemma 2.4 and Theorem
2.5]{Yoshioka:example-reflections} and \cite[Theorem 4.17]{Yoshioka:Brill-Noether}). One advantage in our description is that we need not
distinguish between Gieseker-stable sheaves in $M_H(\ww_r)$ that are locally free versus those that
are just torsion-free: our discussion in the previous sections shows that instead, $M_{\osigma}(\ww_r)$ is
the right moduli space to consider. One can show that $M_{\sigma_+}(\ww_r)$ consists of
shifts $W = V^\vee[1]$ of derived duals of Gieseker-stable sheaves $V$ of appropriate class. Such a derived dual can be a locally
free sheaf (when $V$ is locally free), or a non-trivial complex $W$ with $H^0(W)$ being a
0-dimensional torsion sheaf. The support of $W$ is simultaneously the locus where the corresponding line bundle in $\VrdH$ is not globally
generated, and where $V$ is not locally free.

Moduli spaces of Gieseker-stable sheaves come equipped with ample line bundles constructed via GIT. 
The closest analogue for moduli spaces of Bridgeland-stable objects comes from the
following result:
\begin{PosLem}[\cite{BM:projectivity}]
Let $\sigma$ be a stability condition on $\Db(X)$ for a smooth projective\footnote{See
\cite{BCZ:nef} for a generalisation to singular quasi-projective varieties and moduli spaces of
objects with compact support.} variety $X$,
and assume we are given a family $\cE$ of
$\sigma$-semistable objects parameterised by a variety $S$. Then this induces a real nef divisor
class\footnote{To be precise, when $S$ is singular we obtain a numerical Cartier divisor class.}
$l_\sigma \in \NS(S) \otimes \R$ on $S$. Moreover, for a curve $C \subset S$ we have $l_\sigma.C = 0$ if and only if 
the objects parameterised by $C$ are \emph{S-equivalent} to each other.
\end{PosLem}
Any $\sigma$-semistable object has a \emph{Jordan-H\"older filtration}:
a filtration whose factors are $\sigma$-\emph{stable} of the same slope. 
Two semistable objects are called $S$-equivalent if their Jordan-H\"older filtrations have the same
stable quotients. In practice, this often means that the Positivity Lemma not only produces nef
divisors, but also dually extremal curves describing a boundary facet of the nef cone.

The line bundle can be constructed as follows: we can always normalise the central charge to satisfy
$Z(\vv) = -1$. Via the Mukai pairing, the imaginary part $\Im Z$ of the central charge can be identified with
an element of $\vv^\perp \otimes \R \subset H^*_{\mathrm{alg}}(X, \R)$. Then
$l_\sigma = \theta_v(\Im Z)$, where $\theta_v$ is the Mukai isomorphism of Theorem
\ref{thm:Mukaiisom}.

We now apply the Positivity Lemma in our situation. Let us again fix $\vv = (0, H, d+1-g)$, and assume for simplicity that
the moduli space $M_H(\vv)$ of torsion sheaves has a universal family. (This assumption is
satisfied when $H^2$ and $d + 1-g$ are coprime; otherwise one can descend the line bundle
constructed in the following from an \'etale cover of the moduli space that admits a universal family.)
We now consider its universal family as a family of
$\sigma_0$-\emph{semistable} objects. The Positivity Lemma produces a nef line bundle $l_{\sigma_0}$
on $M_H(\vv)$. Using Remark \ref{rmk:thetapairing}, one can additionally show that its volume is positive,
and hence that $l_{\sigma_0}$ is big. Since $M$ is K-trivial, the base point free theorem says that 
$l_{\sigma_0}$ is globally generated, and so it produces a birational contraction
\[
\phi_0 \colon M_H(\vv) \to \overline{M}.
\]

To understand the contracted locus, we have to understand S-equivalence for objects in $M_H(\vv)$ with
respect to $\sigma_0$. By Lemma \ref{lem:firstwall}, the Jordan-H\"older filtration of $L \in
M_H(\vv)$ is trivial
when $h^0(L) = 0$; otherwise, its filtration quotients are given by $\cO_X$ with multiplicity
$h^0(L)$, and by the quotient $W$ in the short exact sequence \eqref{eq:firstwall}. In other words,
two objects $L, L'$ are S-equivalent if and only if $h^0(L) = h^0(L')$, i.e. 
$L, L' \in \VrdH$ for some $r$, and if they are in the same Grassmannian fibre of the map $\phi$ in
\eqref{maindiagram}. In summary, we have proved (see also \cite[Section 4]{Yoshioka:Brill-Noether}):

\begin{Thm} \label{thm:contractiongeometry}
Assume that $(X, H)$ satisfy assumption (*), that $\vv = (0, H, d+1-g)$ for some
$0 < d \le g-1$. Then the moduli space $M_H(\vv)$ of torsion sheaves admits a birational contraction
$\overline{\phi} \colon M_H(\vv) \to \overline{M}$, whose 
exceptional locus is the Brill-Noether locus $\mathrm{BN}_d(\abs{H})$. The
natural stratification of $\mathrm{BN}_d(\abs{H})$ by the number of global section corresponds to
the stratification induced by $\overline{\phi}$: each stratum is a Grassmannian-bundle over its image
in $\overline{M}$.
\end{Thm}

In this context, Lemma \ref{lem:pullback} becomes a well-known statement,
see e.g.~\cite[Lemma 2.9]{Kaledin:symplectic-singularities}.

\section{Birational geometry of moduli spaces of sheaves: a quick survey}
\label{sect:biratgeneral}

Many of the statements we have shown so far can be proved in much bigger generality: given a K3
surface $X$ and a primitive class $\vv \in \Halg$, one can describe the location of all walls for
$\vv$ in the entire space of stability conditions, and then in turn use that to completely describe
the birational geometry of the moduli space $M_H(\vv)$ of Gieseker-stable sheaves.

The idea is simple. Let $E$ be an object of Mukai vector $\vv$ that is strictly semistable with respect to
a stability condition $\sigma$ on a general point of a given wall. We consider its Jordan-H\"older factors. If $\aa_1, \dots, \aa_m$ are their Mukai vectors, then
\begin{itemize}
\item $\vv = \aa_1 + \dots + \aa_m$;
\item the wall is locally described by the condition that the
central charges $Z(\aa_i)$ all lie on the same ray;
\item by Theorem \ref{thm:dimmodspace} we have $\aa_i^2 \ge -2$ for all $i$; and
\item all $\aa_i$ are contained in a common rank two sublattice of $\Halg$.
\end{itemize}

One can in fact prove the converse: if all four conditions above are satisfied, then the stability
condition lies on a wall for $\vv$. (The main complication comes from ``totally semistable walls'':
there might not exist any object of class $\aa_i$ that is \emph{stable} on the wall; in this case,
we have to use a different decomposition of $\vv$ within the same rank two sublattice.) Further, one
can determine when there exist curves of S-equivalent objects, and thus whether the wall induces a
birational contraction.

This analysis is the main content of \cite{BM:walls}. It leads, for example, to a complete
description of the nef cones of all birational models of $M_H(\vv)$ inside 
$\NS\left(M_H(\vv)\right) \otimes \R$. To explain that description, we first need to recall a few 
basic facts about birational geometry and the Beauville-Bogomolov form on irreducible holomorphic
symplectic varieties. It is a quadratic form on $\NS(M_H(\vv))$ of signature
$(1, \rho-1)$. The cone defined by $(D, D) > 0$ thus has two components; one of them contains the
ample cone, and we will call this component the \emph{positive cone}. The volume of a divisor
$D$ is, up to a constant factor, given by $(D, D)^n$ where $2n = \dim M_H(\vv)$, and thus the cone
of movable divisors is contained in the closure of the positive cone. The cone of movable divisors admits a chamber
decomposition whose chambers correspond one-to-one to smooth, K-trivial birational models 
$g \colon M_H(\vv) \dashrightarrow N$ of
$M_H(\vv)$: the chamber is given as $g^* \mathrm{Nef}(N)$; see \cite{HassettTschinkel:MovingCone}.

\begin{Thm}[{\cite[Theorem 12.1]{BM:walls}}] \label{thm:nefcones}
Inside the positive cone of $M_H(\vv)$, each chamber of the movable cone is cut out by hyperplanes of
the form $\theta_\vv(\vv^\perp \cap \aa^\perp)$ for all $\aa \in \Halg$ satisfying
$\aa^2 \ge -2$ and $\abs{(\vv, \aa)} \le \frac{\vv^2}2$. 
\end{Thm}

In other words, given the arrangement of hyperplanes of the form $\theta_\vv(\vv^\perp \cap \aa^\perp)$
for all $\aa$ as above, each such chamber is a connected component of the complement. Combined with
a similar description of the movable cone (which is due to Markman \cite{Eyal:survey}, but can also
be reproved with the methods discussed here) this leads to a complete list of all birational models of
$M_H(\vv)$; the only necessary ingredient is the Picard lattice of $X$.

In any given example, one can also attempt to study the birational geometry of the contraction in
order to obtain a result analogous to Theorem \ref{thm:contractiongeometry}; this has been done
systematically up to dimension 10 in \cite{HT:survey} (along with other applications).

For an analogue of Theorem \ref{thm:nefcones} for the singular O'Grady spaces of dimension
10, see \cite{Ciaran-Ziyu}. 

\subsection*{Deformations} Using either twistor deformations \cite{Mongardi:note} or deformation
theory of rational curves in families of irreducible holomorphic symplectic manifolds (IHSM) \cite{Mori-cones} one can deform Theorem \ref{thm:nefcones} to an analogue for all
IHSM deformation-equivalent to Hilbert schemes on K3
surfaces; this concludes a programme started in \cite{HassettTschinkel:RationalCurves}. Thus, indirectly, the methods discussed here lead to a description of the birational
geometry of varieties that (currently) have no interpretation as a moduli space. 

\subsection*{Other surfaces} In the case of abelian surfaces, or K3 surfaces of Picard rank one, the
Positivity Lemma was first proved in  \cite{Minamide-Yanagida-Yoshioka:wall-crossing,MYY2} using
Fourier-Mukai transforms. Yoshioka then deduced in \cite{Yoshioka:cones} a
description of nef cones of (Kummer varieties associated to) moduli spaces of sheaves on abelian
surfaces, obtaining a result completely analogous to Theorem \ref{thm:nefcones}.

Extending this result to other surfaces is, to some extent\footnote{The main difficulty specific 
to K3 surfaces is essentially due to the large group of autoequivalences of $\Db(X)$: they produce
many walls where \emph{every} object in a given moduli space becomes strictly semistable. The
easiest example is the analogue of our situation for $d \ge g$: the wall corresponding to Lemma
\ref{lem:firstwall} now destabilises all torsion sheaves. The wall-crossing still induces a
birational transformation of the moduli spaces, but on the common open subset each stable object
gets replaced via its image under the auto-equivalence given by the spherical twist at $\cO_X$.}, much more difficult. Even for
Gieseker-stable sheaves, it is in general unknown for which Chern classes there exist Gieseker-stable
sheaves, i.e.~there is no analogue of Theorem \ref{thm:dimmodspace}. Even when it exists, as in the
case of $\P^2$, the answer \cite{Drezet-LePotier} is quite intricate. Moreover, the answer changes
as we move from Gieseker-stability to Bridgeland stability conditions, making the wall-crossing analysis 
much more of a moving target. 

For an Enriques surfaces $S$, one can circumvent some of these difficulties by using the pull-back
map $\pi^*$ where $\pi \colon X \to S$ is the associated 2:1-covering by a K3 surfaces; this induces
a finite map between corresponding moduli spaces, and can be used to show that the nef divisors
produced by the Positivity Lemma are actually ample.
The results are especially powerful for \emph{unnodal} Enriques surfaces (i.e., not containing a
smooth rational curve); see \cite{Howard:stabEnriques}. 

\subsection*{Projective plane} The entire story originally started with the case of $\P^2$: in \cite{ABCH:MMP}, the authors observed the
correspondence between walls for stability conditions and birational transformations of the Hilbert
scheme of $n$ points on $\P^2$ in many examples---for example, including all walls for all
$n \le 9$; they conjectured the correspondence between stable base loci and destabilised objects in
general. This paper was the original motivation behind all the developments discussed here,
and in particular directly motivated the Positivity Lemma above.

The correspondence of \cite{ABCH:MMP} was generalised to all
Gieseker-moduli spaces and proved in \cite{AaronAndStudents:P2}; a different argument in
\cite{Izzet-Jack:interpolation} treated the case of torus fixed points in the Hilbert scheme. It was
upgraded to a birational correspondence (by proving that all Bridgeland moduli spaces appearing in
the wall-crossing for the Hilbert scheme are irreducible) in \cite{Chunyi-Xiaolei:MMP}. The authors
also extended their results to commutative deformations of $\Hilb^n(\P^2)$ using stability conditions on
the derived category of non-commutative deformations of $\P^2$. 

From this correspondence, one can deduce a description of the nef cone of Gieseker-moduli spaces, 
see \cite{CC:one-dimensional, Matthew:torsion} for torsion sheaves, and 
\cite{Izzet-Jack:ample} for small rank or large discriminant: again the idea is to apply the
Positivity Lemma at a wall, producing a nef divisor \emph{and}, dually, a contracted extremal curve
of S-equivalent objects.

But due to the difficulties hinted at above, it took additional effort to understand the entire
picture, including the nef cones of birational models.  One needed to understand for which wall a
moduli space of stable objects of given Chern character becomes empty. This turns out to be closely
related to another classical problem:

\begin{heur} For any class $\vv \in H^*(\P^2)$, determining the ``last wall'', i.e. the wall
after which $M_{\sigma}(\vv)$ becomes empty, is equivalent to determining the boundary of the
effective cone of $M_H(\vv)$.
\end{heur}

The reasoning behind this heuristic goes as follows. Consider the nef divisor $l_{\osigma}$ given by
the Positivity Lemma for $\osigma$ lying on this ``last
wall''; in particular, this means every object becomes strictly semi-stable with respect $\osigma$.
Then one can \emph{expect} every point in the moduli space to lie on a curve of objects that are
S-equivalent with respect to $\osigma$; in other words, $l_{\osigma}$ is dual to a moving curve in
the Mori cone. This implies that $l_{\osigma}$ is on the boundary of the effective cone.\footnote{This
is a heuristic argument only for two reasons: even if every object is strictly semistable, some or
all of them could be the unique non-trivial extensions in their S-equivalence class. Moreover
when all objects become strictly semistable, that does not a priori preclude the existence of new
stable objects on the other side of the wall; in that case, the wall corresponds to the boundary of
the effective cone, but is not the ``last wall''.}


The problem of determining the effective cone was solved in \cite{Jack:Steiner,
Huizenga:P2} for the Hilbert scheme,  in \cite{CC:one-dimensional, Matthew:torsion} for
one-dimensional torsion sheaves, and in \cite{Izzet-Jack-Matthew} for all Gieseker-moduli spaces;
see \cite{Izzet-Jack:P2} for a survey of the results and the arguments, and the relation to the
interpolation problem. The recent preprint \cite{Chunyi-Xiaolei:birational} then made the above
heuristic reasoning precise, and used it to give a complete description of the decomposition of the
movable cone into chambers corresponding to nef cones of birational models. I would like to explain
one more consequence of their results:

\begin{Prop}
[{\cite[Theorem 0.1]{Chunyi-Xiaolei:MMP}, \cite[Corollary
0.3]{Chunyi-Xiaolei:birational}, building on essentially all the other results mentioned in this
section}] \label{prop:smoothness}
Let $\vv \in H^*(\P^2)$ be a primitive class, let $M(\vv)$ be the moduli space of
Gieseker-stable sheaves of Chern character $\vv$, and let $M \dashrightarrow M(\vv)$ be a birational
model corresponding to an open chamber in the movable cone of $M(\vv)$. Then $M$ is smooth. 
\end{Prop}

To explain the argument, let us briefly recall why $M(\vv)$ is smooth. For 
$F \in M(\vv)$, we have to show $\Ext^2(F, F) = 0$; by Serre duality,
$\Ext^2(F, F) = \Hom(F, F(-3))^\vee$; since $F, F(-3)$ are both slope-semistable with $\mu(F) >
\mu(F(-3))$, the claim follows. 

To generalise this to birational models of $M(\vv)$, we first use their interpretation as moduli spaces. 
As indicated previously, we
know that $M \cong M_\sigma(\vv)$ where $\sigma = \sigma_{\alpha, \beta}$ lies in an open chamber of the space of stability
conditions. As above, for $E \in M_{\sigma}(\vv)$, we have $\Ext^2(E, E) = \Hom(E, E(-3))^\vee$.
However, Bridgeland stability is not invariant under $\blank \otimes \cO(-3)$; instead, all we know
a priori is that $E(-3)$ is $\sigma_{\alpha, \beta-3}$-stable. The key argument of
\cite{Chunyi-Xiaolei:MMP, Chunyi-Xiaolei:birational} now shows that as we follow
the natural path from $\sigma_{\alpha, \beta-3}$ to $\sigma_{\alpha, \beta}$, we can control the
phases of the semistable factors appearing in the Harder-Narasimhan filtration of $E(-3)$, and conclude that they all have
smaller phase than that of $E$; then the Hom-vanishing follows again from stability.

\subsection*{General surfaces}
Similar results for the the Hilbert scheme on other rational surfaces were obtained in
\cite{Aaron-Izzet:pointsonsurfaces}, for example including nef cones of all Hilbert schemes points
on Hirzebruch surfaces. In the case of $\P^1 \times \P^1$, the effective cone of many
moduli spaces of sheaves have been determined in \cite{Tim:effectiveP1P1}, and in all cases where
$c_1$ is symmetric in \cite{Abe:symmetricP1P1}. 

Two recent articles show that one can make at least some of the arguments simultaneously for all
surfaces. 
For example, one of the main results of \cite{7authors:nefcones} shows that for a surface of Picard
rank one and $n \gg 0$, one can determine the nef cone of $\Hilb^n(X)$. The assumption of $n \gg 0$
is needed to ensure that an effective curve $C$ of minimal degree has non-empty $W^1_n(C)$; 
the associated map $C \to \P^1$ produces the curve of S-equivalent objects dual to the nef divisor
class coming from the Positivity Lemma. Similarly, in
\cite{Izzet-Jack:strongBG} the authors show that if one fixes the rank $r$ and the first Chern character
$c$, then for $s \ll 0$ one can determine the nef cone of the moduli space of
Gieseker-stable sheaves on $X$ of Chern character $(r, c, s)$ if one knows the set of Chern classes of
semistable bundles on $X$. (In other words, the assumption $s \ll 0$ circumvents the problem of
knowing when moduli spaces of $\sigma$-stable objects become empty.)

\subsection*{Other applications}
We list a few more relations between stability conditions and classical questions 
that have appeared in the literature, and may lead to more applications in the future.
\begin{itemize}
\item The contraction from the Gieseker-moduli space to the Uhlenbeck space of
slope-semistable vector bundles can be induced by wall-crossing \cite{LoQin:miniwalls,
Jason:Uhlenbeck} (i.e., there is a wall for which the associated line bundle induces this
contraction).
\item Similarly, the Thaddeus-flips constructed in \cite{MatsukiWenthworth:TwistedVariation}
relating Gieseker-moduli spaces for different polarisations are
induced by a sequence of walls \cite{Yoshioka:wallcrossingonblowup, 1505.07091}.
\item Flips of secant varieties can be shown to arise naturally in the wall-crossing for moduli
spaces of torsion sheaves on $\P^2$ \cite{1311.1183}.
\item One can induce the minimal model programme of a surface $X$ via
wall-crossing in $\Db(X)$ \cite{Toda:MMPsurfaces}; yet the moduli space becomes reducible if one tries to
contract other curves of self-intersection less than -2 \cite{Becca:thesis}. 
\item There is a a close relation between the location of the wall where a given ideal
sheaf in $\Hilb^n(\P^2)$ gets destabilised and its Castelnuovo-Mumford regularity
\cite{Izzet-Hyeon-Park}.
\end{itemize}

Some recent developments have already lead to new results.

In \cite{Wafa-Antony}, the authors combine stability conditions with Fourier-Mukai techniques 
to determine precisely which line bundles on an abelian surface of Picard rank one are $k$-\emph{very ample}.

Finally, returning to a topic closely related to the main content of this survey, consider a
globally generated line bundle $L \in \VrdH$, and its Mukai-Lazarsfeld bundle $M_L$ (where $M_L
\cong W[-1]$ with $W$ as given in Lemma \ref{lem:firstwall}). In
\cite{Soheyla:Mercat}, the author uses stability conditions in order to prove ordinary slope-stability of
the restriction of $M_L$ to any curve in $\abs{H}$. This leads to many new counter-examples
to Mercat's conjecture, which was a proposed bound for the analogue of the Clifford index for slope-stable vector
bundles on curves in terms of the Clifford index for line bundles.

\bibliography{all}                      
\bibliographystyle{halpha}     

\end{document}